\pdfoutput=1
\RequirePackage{ifpdf}
\ifpdf 
\documentclass[pdftex]{sigma}
\else
\documentclass{sigma}
\fi

\numberwithin{equation}{section}

\newtheorem{thm}{Theorem}[section]
\newtheorem{cor}[thm]{Corollary}
\newtheorem{Lemma}[thm]{Lemma}
\newtheorem{prop}[thm]{Proposition}
 { \theoremstyle{definition}
\newtheorem{defn}[thm]{Definition}

\newtheorem{rmk}[thm]{Remark} }

\newcommand{\del}{\partial}
\newcommand{\dt}{\frac{\partial}{\partial t}}
\newcommand{\brs}[1]{\left| #1 \right|}

\newcommand{\eps}{\epsilon}
\newcommand{\gD}{\Delta}

\newcommand{\ga}{\alpha}
\newcommand{\gb}{\beta}

\newcommand{\N}{\nabla}
\newcommand{\FF}{\mathcal F}
\newcommand{\WW}{\mathcal W}

\newcommand{\LL}{\mathcal L}

\renewcommand{\part}{\del}

\newcommand{\IP}[1]{\left<#1\right>}
\newcommand{\wnorm}[1]{\left|#1\right|^2_{\frac1k}}
\newcommand{\wwnorm}[1]{\left|#1\right|^2_{\frac{k-1}k}}

\DeclareMathOperator{\Sym}{Sym}
\DeclareMathOperator{\Rc}{Rc}

\DeclareMathOperator{\tr}{tr}

\DeclareMathOperator{\divg}{div}

\begin{document}

\allowdisplaybreaks

\renewcommand{\thefootnote}{}

\newcommand{\arXivNumber}{2306.01649}

\renewcommand{\PaperNumber}{003}

\FirstPageHeading

\ShortArticleName{Optimal Transport and Generalized Ricci Flow}

\ArticleName{Optimal Transport and Generalized Ricci Flow\footnote{This paper is a~contribution to the Special Issue on Differential Geometry Inspired by Mathematical Physics in honor of Jean-Pierre Bourguignon for his 75th birthday. The~full collection is available at \href{https://www.emis.de/journals/SIGMA/Bourguignon.html}{https://www.emis.de/journals/SIGMA/Bourguignon.html}}}

\Author{Eva KOPFER~$^{\rm a}$ and Jeffrey STREETS~$^{\rm b}$}

\AuthorNameForHeading{E.~Kopfer and J.~Streets}

\Address{$^{\rm a)}$~Institut f\"ur Angewandte Mathematik, Universit\"at Bonn, 53115 Bonn, Germany}
\EmailD{\href{mailto:eva.kopfer@iam.uni-bonn.de}{eva.kopfer@iam.uni-bonn.de}}

\Address{$^{\rm b)}$~Rowland Hall, University of California, Irvine, CA, USA}
\EmailD{\href{mailto:jstreets@uci.edu}{jstreets@uci.edu}}

\ArticleDates{Received June 06, 2023, in final form January 06, 2024; Published online January 10, 2024}

\Abstract{We prove results relating the theory of optimal transport and generalized Ricci flow. We define an adapted cost functional for measures using a solution of the associated dilaton flow. This determines a formal notion of geodesics in the space of measures, and we show geodesic convexity of an associated entropy functional. Finally, we show monotonicity of the cost along the backwards heat flow, and use this to give a new proof of the monotonicity of the energy functional along generalized Ricci flow.}


\Classification{53E20; 49Q22}

\renewcommand{\thefootnote}{\arabic{footnote}}
\setcounter{footnote}{0}

\section{Introduction}

The theory of optimal transport plays a key role in our understanding of the geometry of Ricci curvature. In recent years, there have been significant applications to the theory of Ricci flow. The fundamental work of McCann--Topping \cite{McCannTopping} establishes Wasserstein distance monotonicity for measures evolving by the backward heat equation along Ricci flow. Later Topping \cite{ToppingOT} considered a cost associated to Perelman's length functional, and established an entropy convexity formula for certain measures along the flow, using this to give a different proof of the monotonicity of Perelman's entropy functional. Lott \cite{LottOTPRV} extended this in several directions, proving analogous convexity formulas which recover the monotonicity formulas for $\FF$ and $\WW_+$ (see \cite{FIN, Perelman1}), and furthermore using these results to recover the monotonicity of the reduced volume.\looseness=1

Our purpose in this work is to further extend these results to the setting of generalized Ricci flow. This equation couples the Ricci flow to an evolution equation for a closed three-form and dilaton function, arising naturally in and with applications to mathematical physics \cite{HeLiuGRF,OSW,Polchinski,StreetsTdual} and complex geometry \cite{JFS, StreetsPCFBI, PCF}. Furthermore, the equation is closely related to Hitchin's generalized geometry program \cite{AXu,GualtieriGKG,HitchinGCY}, see, e.g., \cite{MGFRicciflowTdual, GRFbook, GKRF}. We will take the point of view in \cite{Streetsscalar} and consider the Ricci flow coupled to a differential form of arbitrary positive degree. To describe the equation, first consider a Riemannian manifold $(M, g)$ and fix $H=\bigoplus_{k=1}^n H_k\in\Lambda^*T^*M$. For this data, we define
\begin{align*}
 H^2 \in \Sym^2 T^*M,\qquad H^2(X,Y) := \IP{i_X H, i_Y H}, \qquad \wwnorm{H}:= \sum_{k=1}^n\frac{k-1}k|H_k|^2.
\end{align*}
Following \cite{Streetsscalar}, we say a one parameter family $(g_t, H_t, f_t)$ of Riemannian metrics, differential forms and smooth functions is a solution of generalized Ricci flow if
\begin{gather*}
\del_tg= -2 \Rc + \frac12 H^2- 2 \N^2 f,\qquad
\del_t H= \Delta_dH-d i_{\N f}H,\\
\del_t f= \Delta f+\frac14 \wwnorm{H}-|\N f|^2.
\end{gather*}
The equation for the function $f$ is called the dilaton flow in \cite{Streetsscalar}, in part due to its appearance in the physical theory of renormalization group flow \cite{Polchinski}.

A fundamental observation about the generalized Ricci flow is that the time-dependent metric is gauge-equivalent to a supersolution of Ricci flow. As noted above, McCann--Topping showed monotonicity of Wasserstein distance for measures evolving by the backwards heat flow along a~supersolution to Ricci flow. Our first result explicitly derives this for generalized Ricci flow (cf.~Corollary~\ref{c:WDM}), with the proof using a notion of the energy of a path of measures which explicitly incorporates the dilaton weight $f$. Next, we extend results of \cite{LottOTPRV,ToppingOT} and define an adapted cost for paths of measures in terms of a solution of the associated continuity equation, where again the associated dilaton flow plays a key role. This cost determines a formal Riemannian geometry on the space of probability measures. There is furthermore a natural entropy for such measures, and our second main result establishes geodesic convexity of this entropy (cf.~Proposition~\ref{prop:entropyconvexity}). We furthermore show that the cost of paths is monotone along the backwards heat flow (cf.~Corollary~\ref{cor: cost}). Finally, we use this to give a new proof of the monotonicity of the $\FF$-functional along generalized Ricci flow (cf.\ Corollary~\ref{c:Fmonotone}).

\section{Background}

In this section, we recall some fundamental results related to the generalized Ricci flow equation. Given a smooth manifold, fix $g$ a Riemannian metric, \[H = \bigoplus_{k=1}^n H_k \in \Lambda^* T^* M,\] and a smooth function $f$. We recall the weighted sum defined in the introduction, and furthermore introduce
\begin{align*}
 \wwnorm{H}:= \sum_{k=1}^n\frac{k-1}k|H_k|^2, \qquad \wnorm{H} := \sum_{k=1}^n \frac{1}{k} \brs{H_k}^2.
\end{align*}
This data also determines notions of Ricci and scalar curvature:
\begin{defn} Given $(g, H, f)$ as above, the \emph{Ricci tensor} is
\begin{align*}
 \Rc^{H,f} := \Rc - \frac{1}{4} H^2 + \N^2 f - \frac{1}{2} \left( d^*_g H + i_{\N f} H \right) \in \Sym^2 T^* M \oplus \bigoplus_{k=0}^{n-1} \Lambda^k T^* M.
\end{align*}
Furthermore, the \emph{scalar curvature} is
\begin{align*}
 R^{H,f} = R - \frac{1}{4} \wnorm{H} + 2 \gD f - \brs{\N f}^2.
\end{align*}
\end{defn}

\begin{rmk} If a superscript in $\Rc^{H,f}$ or $R^{H,f}$ is dropped then the notation refers to the corresponding quantity with that term set to zero, i.e., $\Rc^f = \Rc + \frac{1}{2} \N^2 f$.
\end{rmk}

We note that in the case $H \in \Lambda^3$ and $f$ is constant, the Ricci tensor above is precisely the Ricci tensor of the unique metric-compatible connection with torsion $H$, which is a two-tensor with a symmetric and skew-symmetric part. For $H \in \Lambda^3$ and $f$ arbitrary, this tensor was defined in \cite{Streetsscalar} and named the twisted Bakry--\'Emery tensor. For general $H$ but constant $f$, this tensor is in the spirit of the generalized Ricci tensor used in \cite{FNS}. The coupling of the Ricci tensor to forms of arbitrary degree arises naturally in supergravity theories. Taking a hint from this, it may be possible to describe this Ricci curvature in general in terms of the curvature of a generalized connection on some augmented tangent bundle, as in the case of three-forms and the Bismut connection~\cite{GRFbook}.

For our purposes here, these definitions are justified by a key monotonicity formula for the scalar curvature along generalized Ricci flow. Given $(g_t, H_t, f_t)$ a solution to generalized Ricci flow as described in the introduction, we let
\begin{align*}
\square_f := \dt - \gD_f = \dt - \gD + \N f, \qquad \divg_f X := {\rm e}^{f} \divg \big({\rm e}^{-f} X\big)
\end{align*}
denote the forward weighted heat operator and weighted divergence. Before stating the result, we record some consequences of the fact that $H$ is closed which are left as exercises (cf.\ \cite[Lemma~3.19]{GRFbook} for the case $H$ is a three-form):
\begin{Lemma} \label{l:HBianchi} Given $(g, H)$ as above, one has
\begin{align*}
& \divg H^2 = - \IP{d^* H, H} + \frac{1}{2} d \wnorm{H},\qquad
 \divg \divg H^2 = \frac{1}{2} \gD \wnorm{H} + \sum_{k=1}^n \frac{1}{k} \IP{\gD_d H, H} + \brs{d^*_g H}^2,
\end{align*}
where for a $(k-1)$-form $\ga$ and $k$-form $\gb$, the notation $\IP{\ga, \gb}$ denotes the $1$-form uniquely defined by $\IP{\ga,\gb}(X) = \IP{\ga, i_X \gb}$.
\end{Lemma}

\begin{prop}[{\cite[Proposition 2.11]{Streetsscalar}}]\label{p:scalarmonotonicity} Given $(g_t, H_t, f_t)$ a solution to generalized Ricci flow, one has
\begin{align*}
\square_f R^{H,f} = 2 \big|\Rc^{H,f}\big|^2.
\end{align*}
\begin{proof} The result is claimed in \cite{Streetsscalar} without proof, so we include the short calculation here for convenience. Note furthermore that we are working here with the flow modified by diffeomorphisms generated by $\N f$. We compute the time derivative of each term in $R^{H,f}$ separately.
First, we compute that
	\begin{align}
 	\del_t R={}&-\IP{\Rc,\del_tg}+\divg\divg\del_t g-\Delta(\tr \del_t g)\nonumber\\
 ={}& 2\IP{\Rc,\Rc-\frac14 H^2}+\Delta R+\frac12\divg\divg H^2-\IP{\N R,\N f}-\frac12\Delta|H|^2, \label{eq: scalar}
 \end{align}
 	where we used the Bianchi identity and that
 	\begin{align*}
 	&\divg \N^2f= \N\Delta f+\Rc(\N f),\qquad\divg\divg \N^2f= \Delta^2 f+\frac12\IP{\N R,\N f}+\IP{\Rc,\N^2 f}.
 	\end{align*}
 	Then, we observe using Bochner's formula
 	\begin{align}
 	\del_t|\N f|^2={}& 2\left(\Rc+\N^2 f-\frac14 H^2\right)(\N f,\N f)+2\IP{\N f,\N\left( \Delta f-|\N f|^2 + \frac{1}{4} \wwnorm{H} \right)}\nonumber\\
 ={}& \Delta|\N f|^2-2|\N^2f|^2-\IP{\N f,\N |\N f|^2}+\frac12\big\langle\N\wwnorm{H},\N f\big\rangle\nonumber\\
 &-\frac12\IP{H^2,\N f\otimes \N f}.\label{eq: gradient}
 \end{align}
 Next, we compute 	
 	\begin{align}
 	\del_t \Delta f={}&\Delta \del_t f-\IP{\del_t g,\N^2 f}-\IP{\divg(\del_tg)-\frac12\N(\tr\del_tg),\N f}\nonumber\\
 ={}&\Delta^2 f-\Delta|\N f|^2+2\IP{\Rc+\N^2 f-\frac14 H^2,\N^2 f}-2\IP{-\divg\N^2f+\frac12\N\Delta f,\N f}\nonumber\\
 	& +\IP{-\frac12\divg H^2+\frac14\N |H|^2,\N f}+\frac14\Delta \wwnorm{H}.\label{eq: laplace}
 \end{align}
 	Finally, one has easily
 	\begin{align} \label{eq: h}
 	\del_t |H_k|^2=-k\IP{\del_t g, H_k^2}+2\IP{H_k,\Delta_d H_k-di_{\N f} H_k}.
 	\end{align}
 	Inserting \eqref{eq: scalar}, \eqref{eq: gradient}, \eqref{eq: laplace} and \eqref{eq: h} into the definition of $R^{H,f}$ yields
 \begin{align*}
 	\del_tR^{H,f}
 	={}&\Delta R^{H,f}+2\big|\Rc^{H,f}\big|^2-\big\langle\N R^{H,f},\N f\big\rangle,
 	\end{align*}
 where we used the identities for $\divg \N^2 f$ above and Lemma \ref{l:HBianchi}. The proposition follows.
\end{proof}
\end{prop}

\section{Wasserstein distance monotonicity for generalized Ricci flow}

Given a smooth connected manifold $M$, let $P(M)$ denote the space of Borel probability measures with finite second moments, i.e.,
\begin{align*}
 P(M):=\left\{\mu \text{ Borel probability measure}: \int_Md^2(x,x_0) {\rm d}\mu(x)<\infty \text{ for some }x_0\in M\right\}.
\end{align*} This space is naturally endowed with the Wasserstein distance $W$, defined for $\mu_1,\mu_2\in P(M)$ by the optimal transport problem
\begin{align*}
 W(\mu_1,\mu_2)^2:=\inf\int_{M\times M}d^2(x,y) {\rm d}\gamma(x,y),
\end{align*}
where the infimum is taken over all couplings $\gamma\in P(M\times M)$ with marginals $\gamma(\cdot\times M)=\mu_1$ and $\gamma(M\times \cdot)=\mu_2$. In the case $(M, g, {\rm e}^{-f} {\rm d}V)$ is a weighted Riemannian manifold, it is useful to consider the subspace $P^\infty(M)\subset P(M)$ consisting of smooth positive densities with respect to the weighted volume measure
\begin{align*}
 P^\infty(M):=\bigl\{\mu\in P(M)\colon {\rm d}\mu=\rho {\rm e}^{-f} {\rm d}V, \rho\in C^\infty(M),\ \rho>0\bigr\}.
\end{align*}
If $\mu\colon[0,1]\to P^\infty(M)$ is a smooth path, we write
\begin{align*}
 {\rm d}\mu(s)=\rho(s) {\rm e}^{-f} {\rm d}V
\end{align*}
and define $\phi(s)$ as a solution to the continuity equation
\begin{align}\label{eq: conteqstatic}
 \del_s\rho=-\divg_f(\rho\N\phi).
\end{align}
Such a $\phi(s)$ exists and is unique up to an additive constant.
Thus for such a path, we may define the Lagrangian
\begin{align*}
 E(\mu)=\frac12\int_0^1\int_M |\N\phi|^2 {\rm d}\mu {\rm d}s.
\end{align*}
A result known as the Benamou--Brenier formula shows that this formal notion of the length of a~path can be used to recover the Wasserstein distance, in the following sense (see \cite[Proposition~4.3]{OW05}):
\begin{thm}\label{thm: existence}
Let $\mu_1,\mu_2\in P^\infty(M)$ be probability measures. Then the infimum of $E$ over smooth curves in $P^\infty(M)$ satisfying the continuity equation and connecting these probability measures is $\frac12 W(\mu_1,\mu_2)^2$.
\end{thm}

In this section, we will analyze the monotonicity of the Wasserstein distance between two backward heat flows of probability measures under generalized Ricci flow. To begin, we record a~fundamental lemma, whose proof is elementary and left to the reader:

\begin{Lemma} \label{l:volumeev} Given $(g_t, H_t, f_t)$ a solution to generalized Ricci flow, one has
\begin{align*}
 \frac{{\rm d}}{{\rm d}t} {\rm e}^{-f} {\rm d}V = - R^{H,f} {\rm e}^{-f} {\rm d}V.
\end{align*}
\end{Lemma}

Also we derive a preliminary computation varying a certain integral along a curve of measures in a fixed time-slice.

\begin{Lemma}\label{lemma: staticderivative} Let $(\rho(s,t),\phi(s,t))_{[0,1]\times [t_0-\eps,t_0+\eps]}$ be a smooth two-parameter family of curves satisfying \eqref{eq: conteqstatic}.
Then for any fixed $t$, we have
 \begin{align*}
& \frac{{\rm d}}{{\rm d}s}\int_M \IP{\N\phi,\N\rho}{\rm e}^{-f} {\rm d}V\\
&\quad =\int_M \left[ - \left(\del_s\phi+\frac12|\N\phi|^2\right)\Delta_f\rho +\big|\N^2\phi\big|^2 \rho + \Rc^f(\N\phi,\N\phi) \rho \right] {\rm e}^{-f} {\rm d}V.
 \end{align*}
\end{Lemma}
\begin{proof}
 Note that
 \begin{align*}
 \frac{{\rm d}}{{\rm d}s}\int_M & \IP{\N\phi,\N\rho}{\rm e}^{-f} {\rm d}V\\
 ={}&\int_M \left[ \IP{\N\left(\del_s\phi+\frac12|\N\phi|^2\right),\N\rho} - \frac12\IP{\N|\N\phi|^2,\N\rho} \right.\\
 &\left.+ \IP{\N\phi,\N(-\divg_f(\rho\N\phi))} \right] {\rm e}^{-f} {\rm d}V\\
 =& \int_M \left[ - \left(\del_s\phi+\frac12|\N\phi|^2\right)\Delta_f\rho + \frac12 \Delta_f|\N\phi|^2 \rho - \IP{\N\Delta_f\phi,\N\phi}\rho \right] {\rm e}^{-f} {\rm d}V.
 \end{align*}
 We obtain the result by applying the weighted Bochner identity \cite[Proposition 3]{BE}:
 \begin{align*}
 \frac12\Delta_f|\N\phi|^2-\IP{\N\Delta_f\phi,\N\phi}=\big|\N^2\phi\big|^2+\Rc^f(\N\phi,\N\phi).
\tag*{\qed}
 \end{align*} \renewcommand{\qed}{}
\end{proof}

Now, we compute the time-derivative of the Lagrangian $E$ of a one-parameter family of curves in $P^\infty(M)$ along generalized Ricci flow.

\begin{prop}\label{prop: energy}
Let $(g_t,H_t,f_t)$ be a generalized Ricci flow for $t\in[t_0-\eps,t_0+\eps]$.
Let $(\rho(s,t),\phi(s,t))_{[0,1]\times [t_0-\eps,t_0+\eps]}$ be a smooth two-parameter family of curves solving \eqref{eq: conteqstatic}.
Let
\begin{align*}
 E(t):=E(\mu(\cdot,t))=\frac12\int_0^1\int_M |\N\phi(s,t)|^2 {\rm d}\mu(s,t) {\rm d}s,
\end{align*}
where $\mu(\cdot,t):=\rho(\cdot,t) {\rm e}^{-f_t} {\rm d}V_t$.
Then
\begin{align*}
\left. \frac{{\rm d}}{{\rm d}t}\right|_{t=t_0}E(t)
={}&\left. \int_M \phi\big(\del_t\rho+\Delta_f \rho-R^{H,f}\rho\big) {\rm e}^{-f} {\rm d}V\right|_{s=0}^1\\
& +\int_0^1\int_M\left[ \big|\N^2\phi\big|^2\rho+\frac14 H^2(\N\phi,\N\phi)\rho\right.\\
&\left.-\left(\del_s\phi+\frac12|\N\phi|^2\right)\big(\del_t\rho+\Delta_f\rho-R^{H,f}\rho\big) \right] {\rm e}^{-f} {\rm d}V {\rm d}s.
\end{align*}
\end{prop}
\begin{proof}
 Using the generalized Ricci flow equations and Lemma \ref{l:volumeev}, we have
 \begin{align*}
\left. \frac{{\rm d}}{{\rm d}t}\right|_{t=t_0} E(t) ={}&\int_0^1\int_M \left[ \Rc^{H,f}(\N\phi,\N\phi)\rho +\IP{\N\phi,\N\del_t\phi}\rho\right.\\
& \left.+\frac12|\N\phi|^2\del_t\rho-\frac12R^{H,f}|\N\phi|^2\rho \right] {\rm e}^{-f} {\rm d}V {\rm d}s.
\end{align*}
 For a fixed $\psi\in C^{\infty}(M)$, we have
 \begin{align*}
 \int_M \psi\del_s\rho {\rm e}^{-f} {\rm d}V=\int_M \IP{\N\psi,\N\phi}\rho {\rm e}^{-f} {\rm d}V.
 \end{align*}
 Hence, integrating by parts in $t$,
 \begin{align*}
 &\int_M \psi\big(\del_s\del_t\rho-R^{H,f}\del_s\rho\big){\rm e}^{-f} {\rm d}V\\
 &\quad= \int_M \big[ 2\Rc^{H,f}(\N\psi,\N\phi)\rho +\IP{\N\psi,\N\del_t\phi}\rho +\IP{\N\psi,\N\phi}\del_t\rho-R^{H,f}\IP{\N\psi,\N\phi}\rho \big] {\rm e}^{-f} {\rm d}V.
 \end{align*}
 For $\psi=\phi$, this yields
 \begin{align*}
& \int_M \phi\big(\del_s\del_t\rho-R^{H,f}\del_s\rho\big){\rm e}^{-f} {\rm d}V\\
 &\quad=
 \int_M \big[ 2\Rc^{H,f}(\N\phi,\N\phi)\rho +\IP{\N\phi,\N\del_t\phi}\rho +\IP{\N\phi,\N\phi}\del_t\rho-R^{H,f}\IP{\N\phi,\N\phi}\rho \big] {\rm e}^{-f} {\rm d}V.
 \end{align*}
 Inserting this into the derivative of $E$ and integrating by parts in $s$ produces
 \begin{align*}
 \frac{{\rm d}}{{\rm d}t} E={}&\int_0^1\int_M \phi\big(\del_s\del_t\rho-R^{H,f}\del_s\rho\big){\rm e}^{-f} {\rm d}V {\rm d}s\\
 &+\int_0^1\int_M \left[ -\Rc^{H,f}(\N\phi,\N\phi)\rho -\frac12|\N\phi|^2\del_t\rho+\frac12R^{H,f}|\N\phi|^2\rho \right] {\rm e}^{-f} {\rm d}V {\rm d}s\\
={}& \left.\int_M \phi\del_t\rho {\rm e}^{-f} {\rm d}V \right|_{s=0}^1 +\int_0^1\int_M \left[ -R^{H,f}\phi\del_s\rho-\Rc^{H,f}(\N\phi,\N\phi)\rho\right.\\
 &\left.-\left(\del_s\phi+\frac12|\N\phi|^2\right)\del_t\rho+\frac12R^{H,f}|\N\phi|^2\rho \right] {\rm e}^{-f} {\rm d}V {\rm d}s.
 \end{align*}
 Note that by Lemma \ref{lemma: staticderivative},
 \begin{align*}
 -\left.\int_M \phi\Delta_f\rho {\rm e}^{-f} {\rm d}V\right|_{s=0}^1
 ={}&\int_0^1\int_M \left[ \big|\N^2\phi\big|^2 \rho +\Rc^f(\N\phi,\N\phi)\rho\right. \\
 &\left.- \left(\del_s\phi+ \frac12|\N\phi|^2\right)\Delta_f\rho \right] {\rm e}^{-f} {\rm d}V,
 \end{align*}
 and
 \begin{align*}
 \frac{{\rm d}}{{\rm d}s}\int_M R^{H,f}\phi\rho {\rm e}^{-f} {\rm d}V=\int_M R^{H,f}\del_s\phi\rho {\rm e}^{-f} {\rm d}V+\int_M R^{H,f}\phi\del_s\rho {\rm e}^{-f} {\rm d}V.
 \end{align*}
 So, combining the above computations gives
 \begin{align*}
 \frac{{\rm d}}{{\rm d}t} E
={}&\left. \int_M \phi\big(\del_t\rho+\Delta_f \rho-R^{H,f}\rho\big) {\rm e}^{-f} {\rm d}V\right|_{s=0}^1+\int_0^1\int_M\left[ \big|\N^2\phi\big|^2\rho+\frac14 H^2(\N\phi,\N\phi)\rho\right.\\
 &\left.-\left(\del_s\phi+\frac12|\N\phi|^2\right)\big(\del_t\rho+\Delta_f\rho-R^{H,f}\rho\big) \right] {\rm e}^{-f} {\rm d}V {\rm d}s,
 \end{align*}
 as claimed.
\end{proof}

As a corollary from this proposition, we obtain the Wasserstein contraction of the backward heat flow of two probability measures under generalized Ricci flow. Here, we denote by $W_t$ the Wasserstein distance associated to time $t$. We first record an elementary lemma showing an equivalent formulation of the backward heat equation in terms of the density, whose proof is left to the reader:

\begin{Lemma} \label{l:rhoBHE} Given $(g_t, H_t, f_t)$ a solution to generalized Ricci flow, suppose $(\mu_t)\subset P^\infty(M)$ is a smooth one-parameter family of probability measures with $\mu_t = \rho_t {\rm e}^{-f_t} {\rm d}V_t$. Then $(\mu_t)$ satisfies the backwards heat flow
\begin{align}\label{eq: heatmeasure}
 \del_t\mu=-\Delta \mu
\end{align}
if and only if
\begin{align}\label{eq: heatdensity}
 \del_t\rho=-\Delta_f\rho + R^{H,f} \rho.
\end{align}
\end{Lemma}

\begin{cor} \label{c:WDM} Let $\big(\mu^1_t\big)$, $\big(\mu^2_t\big)$ be two solutions of the backward heat equation \eqref{eq: heatmeasure}
in $P^\infty(M)$. Then $W_{t}\big(\mu^1_t,\mu^2_t\big)$ is nondecreasing in $t$.
\begin{proof}
 Fix $t_0$. For each $\varepsilon>0$, we may choose according to Theorem \ref{thm: existence} a curve $\mu\colon [0,1]\to P^\infty(M)$ with $\mu(0)=\mu^1_{t_0}$ and $\mu(1)=\mu^2_{t_0}$ satisfying
 \begin{align*}
 E(\mu)\leq \frac12 W_{t_0}\big(\mu^1_{t_0},\mu^2_{t_0}\big)^2+\varepsilon,
 \end{align*}
 where $E(\mu)$ is the Lagrangian of the curve $\mu$ at time $t_0$.
 Let $t\leq t_0$ and let $\mu_t(s)$ be the backward heat flow with $\mu_{t_0}(s)=\mu(s)$. Observe that this implicitly defines two-parameter families $(\rho(s,t), \phi(s,t))$ as described above. Then we know by Proposition \ref{prop: energy} and Lemma \ref{l:rhoBHE} that
 \begin{align*}
 \frac12 W_{t}\big(\mu^1_t,\mu^2_t\big)^2\leq E(\mu_t)\leq E(\mu_{t_0})\leq \frac12 W_{t_0}\big(\mu^1_{t_0},\mu^2_{t_0}\big)^2+\varepsilon.
 \end{align*}
 As $\varepsilon > 0$ is arbitrary, the result follows.
\end{proof}
\end{cor}

\section{Adapted cost for generalized Ricci flow}

In this section, we define a cost adapted to generalized Ricci flow akin to the $\LL_0$-cost in Ricci flow \cite{LottOTPRV}. We will show monotonicity of the cost along the weighted backwards heat equation, and furthermore use this to recapture the monotonicity of the $\FF$-functional. Fix $(g_t, H_t, f_t)$ a solution to generalized Ricci flow on $[0,T]$. Given $\mu_t$ a smooth one-parameter family of probability measures in $P^\infty(M)$ which have densities $\rho_t$ with respect to ${\rm e}^{-f_t} {\rm d}V_t$, it follows that there exists a smooth family $\phi_t$ such that
\begin{align} \label{eq: conteq}
 \del_t\rho=-\divg_f(\rho\N\phi)+R^{H,f}\rho.
\end{align}
For such paths $\mu$ defined on $[t',t''] \subset [0,T]$, we define the Lagrangian
\begin{align*}
E_0(\mu):=\frac12\int_{t'}^{t''}\int_M \big[ |\N\phi|^2+R^{H,f} \big] {\rm d}\mu {\rm d}t.
\end{align*}
This functional can be interpreted as an optimal transport cost for a length functional modified by integrating the weighted scalar curvature $R^{H,f}$ along the curve. This choice is natural given the gradient flow interpretation of generalized Ricci flow \cite{OSW}.

\subsection{Geodesic entropy convexity}

In this subsection, we prove a convexity property for a natural entropy associated to the cost functional $E_0$. We first derive the geodesic equation associated to this cost, then show convexity of the entropy along these geodesics.

\begin{Lemma}\label{lemma: el} Let $(g_t,H_t, f_t)$ be a~solution to generalized Ricci flow. Let
$(\rho(t,s), \phi(t,s))$ be a~two-parameter family of densities and functions satisfying \eqref{eq: conteq}. Then
	\begin{align*}
	\frac{{\rm d}}{{\rm d}s}E_0(\mu(\cdot,s))=\left.\int_M\phi\partial_s\rho {\rm e}^{-f} {\rm d}V\right|_{t=t'}^{t''}
	-\int_{t'}^{t''}\int_M\left[\partial_t\phi+\frac12|\nabla\phi|^2 -\frac12 R^{H,f}\right] \partial_s\rho {\rm e}^{-f} {\rm d}V{\rm d}t.
	\end{align*}
 In particular, a one-parameter $(\rho(t), \phi(t))$ is a geodesic if and only if
\begin{equation}\label{eq: geodesic}
\del_t \rho=-\divg_f(\rho\N\phi)+R^{H,f}\rho,\qquad
\del_t\phi=-\frac12|\N\phi|^2+\frac12 R^{H,f}.
\end{equation}
\end{Lemma}
\begin{proof}
	First of all, we compute
	\begin{align*}
	\frac{{\rm d}}{{\rm d}s}E_0(\mu(\cdot,s)) ={}&\int_{t'}^{t''} \int_M \left[ \IP{\N \phi, \N \partial_s \phi} \rho + \frac{1}{2}\big( \brs{\N \phi}^2 + R^{H,f} \big) \partial_s \rho \right] {\rm e}^{-f} {\rm d}V.
	\end{align*}
	Observe that for an arbitrary function $\psi$, we have by integration by parts
	\begin{align*}
	\int_M \psi {\del_t \rho} {\rm e}^{-f} {\rm d}V
	={}&\int_M \big[ \IP{\N \psi, \N \phi} + \psi R^{H,f} \big] {\rm d}\mu.
	\end{align*}
	It follows that
	\begin{align*}
	\int_M \psi {\del_s\del_t \rho} {\rm e}^{-f} {\rm d}V ={}&\int_M \big[ \IP{\N \psi, \N {\del_s \phi}} \rho + \IP{\N \psi, \N \phi} {\del_s \rho} + \psi R^{H,f}\del_s\rho \big] {\rm e}^{-f} {\rm d}V.
	\end{align*}
	We choose $\psi = \phi$ to yield
	\begin{align*}
	\int_M \phi {\del_s\del_t \rho} {\rm e}^{-f} {\rm d}V ={}&\int_M \big[ \IP{\N \phi, \N {\del_s \phi}} \rho + |\N \phi|^2 {\del_s \rho} + \phi R^{H,f}\del_s\rho \big] {\rm e}^{-f} {\rm d}V.
	\end{align*}
	Combining the above discussion produces
	\begin{align*}
	\frac{{\rm d}}{{\rm d}s}E_0(\mu(\cdot,s))=\int_{t'}^{t''}\int_M\left[ \phi \del_s\del_t\rho -\frac12|\N\phi|^2 \del_s\rho-\phi R^{H,f} \del_s\rho+\frac12 R^{H,f} \del_s\rho\right] {\rm e}^{-f} {\rm d}V{\rm d}t.
	\end{align*}
	Note that Lemma \ref{l:volumeev} further implies
	\begin{align*}
	\del_t\big(\phi \del_s\rho {\rm e}^{-f} {\rm d}V\big)=\big(\del_t\phi \del_s\rho +\phi \del_s\del_t\rho - R^{H,f}\phi \del_s\rho \big){\rm e}^{-f} {\rm d}V.
	\end{align*}
	Consequently, we obtain
	\begin{align*}
	\frac{{\rm d}}{{\rm d}s} E_0(\mu(\cdot,s))={}& \int_{t'}^{t''} \del_t \left[ \int_M \phi\partial_s\rho {\rm e}^{-f} {\rm d}V \right] {\rm d}t \\
 &-\int_{t'}^{t''}\int_M\left[\del_t\phi+\frac12|\N\phi|^2-\frac12 R^{H,f}\right]\del_s\rho {\rm e}^{-f} {\rm d}V{\rm d}t,
	\end{align*}
	which is, after integrating the first term in time, the claim.
\end{proof}

Next, we show the geodesic convexity of a natural entropy quantity associated to this cost. First, we prove two propositions containing useful evolution equations for geodesics along a~solution to generalized Ricci flow.

\begin{Lemma}\label{lem: evolution entropy 1} Fix $(g_t, H_t, f_t)$ a solution to generalized Ricci flow, and suppose $(\rho_t,\phi_t)$ solves the geodesic equations \eqref{eq: geodesic}. Then
	\begin{align*}
	&\frac{{\rm d}}{{\rm d}t}\int_M\phi {\rm d}\mu=\frac12\int_M\big[|\N \phi|^2+R^{H,f}\big] {\rm d}\mu,\\
&	\frac12	\frac{{\rm d}}{{\rm d}t}\int_M|\N\phi|^2 {\rm d}\mu=\int_M \left[\Rc^{H,f}(\N\phi,\N\phi) + \frac12\big\langle \N\phi,\N R^{H,f}\big\rangle\right] {\rm d}\mu.
	\end{align*}
\end{Lemma}

\begin{proof}
	We compute, using the geodesic equation and Lemma \ref{l:volumeev},
	\begin{align*}
	\frac{{\rm d}}{{\rm d}t}\int_M\phi {\rm d}\mu =&\int_M\left[\left(-\frac12|\N\phi|^2+\frac12 R^{H,f}\right)+|\N\phi|^2+\big(R^{H,f}-R^{H,f}\big)\phi\right] {\rm d}\mu\\
	=&\int_M\left[\frac12|\N\phi|^2+\frac12 R^{H,f}\right] {\rm d}\mu,
	\end{align*}
	which yields the first claim. For the second claim, we compute first of all
	\begin{align*}
	\frac{{\rm d}}{{\rm d}t}\frac12|\N\phi|^2={}&\Rc^{H,f}(\N\phi,\N\phi) +\IP{\N\phi,\N\left( -\frac12|\N\phi|^2+\frac12R^{H,f} \right)}.
	\end{align*}
	Hence
	\begin{align*}
	\frac{{\rm d}}{{\rm d}t}\frac12\int_M |\N\phi|^2 {\rm d}\mu=&\int_M \left[\Rc^{H,f}(\N\phi,\N\phi) +\IP{\N\phi,\N\left( -\frac12|\N\phi|^2+\frac12R^{H,f} \right)}\right] {\rm d}\mu\\
	&\ +\int_M \frac12|\N\phi|^2 \left(-\divg_f(\rho\N\phi) \right) {\rm e}^{-f} {\rm d}V\\
	=&\int_M \left[\Rc^{H,f}(\N\phi,\N\phi) +\IP{\N\phi,\frac12R^{H,f}}\right] {\rm d}\mu,
	\end{align*}
 as claimed.
\end{proof}

\begin{Lemma}\label{lem: evolution entropy 2} Fix $(g_t, H_t, f_t)$ a solution to generalized Ricci flow, and suppose $(\rho_t,\phi_t)$ solves the geodesic equations \eqref{eq: geodesic}. Then
	\begin{align*}
	&\frac{{\rm d}}{{\rm d}t} \int_M \log\rho {\rm d}\mu= \int_M \big[\IP{\N\rho,\N\phi}+R^{H,f}\rho \big]{\rm e}^{-f} {\rm d}V,\\
	&\frac{{\rm d}}{{\rm d}t} \int_M \IP{\N\rho,\N\phi}{\rm e}^{-f} {\rm d}V=\int_M \big[|\N^2\phi|^2+\Rc^f(\N\phi,\N\phi)-2\big\langle\Rc^{H,f},\N^2\phi\big\rangle \big] {\rm d}\mu\\
	&\phantom{\frac{{\rm d}}{{\rm d}t} \int_M \IP{\N\rho,\N\phi}{\rm e}^{-f} {\rm d}V=}{} +\int_M \left[ \IP{\frac12\divg H^2 - \frac1{4}\N\wnorm{H},\N \phi} - \frac12 H^2(\N f, \N \phi) \right] {\rm d}\mu\\
	&\phantom{\frac{{\rm d}}{{\rm d}t} \int_M \IP{\N\rho,\N\phi}{\rm e}^{-f} {\rm d}V=}{} +\frac12\int_M \big\langle\N\rho,\N R^{H,f}\big\rangle{\rm e}^{-f} {\rm d}V,\\
	&\frac{{\rm d}}{{\rm d}t}\int_M R^{H,f} {\rm d}\mu
	=\int_M \big[ \del_t R^{H,f} +\big\langle\N R^{H,f},\N\phi\big\rangle \big] {\rm d}\mu.
	\end{align*}
\end{Lemma}
\begin{proof}
	We show the first claim by noting
	\begin{align*}
	\frac{{\rm d}}{{\rm d}t} \int_M \log\rho {\rm d}\mu
	=&\int_M \big[ (\log\rho+1)\bigl(-\divg_f(\rho \N\phi)+R^{H,f}\rho\bigr)-\rho\log\rho R^{H,f} \big] {\rm e}^{-f} {\rm d}V\\
	=&\int_M \big[ \IP{\N\rho,\N\phi}+R^{H,f}\rho \big]{\rm e}^{-f} {\rm d}V.
	\end{align*}
	To show the second claim we will need to decompose the Ricci tensor into its symmetric piece~$\Rc^{H,f}_s$ and anti-symmetric piece $\Rc^{H,f}_a$ (which in general is a polyform). We first compute
	\begin{align*}
	&\frac{{\rm d}}{{\rm d}t} \int_M \IP{\N\rho,\N\phi}{\rm e}^{-f} {\rm d}V\\
&\quad=\int_M \big[ 2\Rc^{H,f}_s(\N\rho,\N\phi) + \big\langle\N\bigl(- \divg_f (\rho \N \phi) +R^{H,f}\rho\bigr),\N \phi\big\rangle \big] {\rm e}^{-f} {\rm d}V \\
	&\quad\phantom{={}}{}+\int_M \left[ \IP{\N\rho,\N \left(-\frac12|\N\phi|^2+\frac12R^{H,f} \right)} - \IP{\N\rho,\N\phi} R^{H,f} \right] {\rm e}^{-f} {\rm d}V.
	\end{align*}
 Using Lemma \ref{l:HBianchi}, we have the Bianchi identity
 \begin{align*}
 	 2\divg \Rc^{H,f}_s=\N R^{H,f}+\N|\N f|^2+\frac1{4}\N\wnorm{H}+2\Rc(\N f)-\frac12\divg H^2,
	\end{align*}
 Using this, we integrate by parts to yield
 \begin{align*}
	&\int_M 2 \Rc^{H,f}_s (\N\rho,\N\phi){\rm e}^{-f} {\rm d}V\\
 &\quad=-2\int_M \big(\big\langle\divg \Rc^{H,f}_s,\N\phi\big\rangle+\big\langle\Rc^{H,f},\N^2\phi\big\rangle- \Rc^{H,f}_s(\N f,\N\phi) \big) {\rm d}\mu\\
	&\quad=-\int_M \left(\IP{\N R^{H,f}+\frac1{4}\N\wnorm{H}-\frac12\divg H^2,\N\phi}\right.\\
&\phantom{\quad={}}{}\left.+2\big\langle\Rc^{H,f},\N^2\phi\big\rangle+\frac12 H^2(\N f,\N\phi) \right) {\rm d}\mu.
\end{align*}
 Further by integration by parts and Bochner's formula
 \begin{align*}
	&\int_M \big\langle\N\big(-{\rm e}^{f}\divg\big(\rho {\rm e}^{-f}\N\phi\big)\big),\N \phi\big\rangle{\rm e}^{-f} {\rm d}V +\int_M \IP{\N\rho,\N\left(-\frac12|\N\phi|^2\right)}{\rm e}^{-f} {\rm d}V\\
	&\quad=-\int_M \IP{\N\phi,\N\Delta\phi} {\rm d}\mu+\frac12\int_M \Delta|\N\phi|^2 {\rm d}\mu +\int_M \IP{\N \phi,\N\IP{\N f,\N\phi}} {\rm d}V\\
 &\quad \phantom{={}}{} -\frac12\int_M \big\langle\N|\N\phi|^2,\N f\big\rangle {\rm d}\mu\\
	&\quad=\int_M \big(\big|\N^2\phi\big|^2+\Rc(\N\phi,\N\phi)\big) {\rm d}\mu +\int_M \IP{\N \phi,\N\IP{\N f,\N\phi}} {\rm d}\mu-\frac12\int_M \IP{\N|\N\phi|^2,\N f} {\rm d}\mu.
	\end{align*}
Consequently,
	\begin{align*}
	&\frac{{\rm d}}{{\rm d}t}\int_M \IP{\N\rho,\N\phi}{\rm e}^{-f} {\rm d}V\\
	&\quad=-\int_M \left(\IP{\N R^{H,f}+\frac1{4}\N\wnorm{H}-\frac12\divg H^2,\N\phi}+2\big\langle\Rc^{H,f},\N^2\phi\big\rangle+\frac12H^2(\N f,\N\phi) \right) {\rm d}\mu\\
	& \phantom{\quad={}}{} +\int_M \left(\big|\N^2\phi\big|^2+\Rc(\N\phi,\N\phi) + \IP{\N \phi,\N\IP{\N f,\N\phi}} -\frac12 \IP{\N|\N\phi|^2,\N f} \right) {\rm d}\mu\\
	& \phantom{\quad={}}{} +\int_M \big\langle\N\big(R^{H,f}\rho\big),\N \phi\big\rangle{\rm e}^{-f} {\rm d}V +\int_M \IP{\N\rho,\frac12\N R^{H,f}}{\rm e}^{-f} {\rm d}V \\
&\phantom{\quad={}}{}-\int_M \IP{\N\rho,\N\phi} R^{H,f} {\rm e}^{-f} {\rm d}V.
	\end{align*}
 Reordering terms yields
	\begin{align*}
	&\frac{{\rm d}}{{\rm d}t} \int_M \IP{\N\rho,\N\phi}{\rm e}^{-f} {\rm d}V\\
	&\quad=\int_M \big(\big|\N^2\phi\big|^2+\Rc(\N\phi,\N\phi)-2\big\langle\Rc^{H,f},\N^2\phi\big\rangle \big) {\rm d}\mu\\
	&\phantom{\quad={}}{}+\int_M \IP{\frac12\divg H^2-\frac1{4}\N\wnorm{H},\N \phi}\rho {\rm e}^{-f} {\rm d}V -\frac12\int_M H^2(\N f,\N\phi) {\rm d}\mu\\
 &\phantom{\quad={}}{}+\frac12\int_M \big\langle\N\rho,\N R^{H,f}\big\rangle{\rm e}^{-f} {\rm d}V
	+\int_M \IP{\N \phi,\N\IP{\N f,\N\phi}} {\rm d}\mu-\frac12\int_M \IP{\N|\N\phi|^2,\N f} {\rm d}\mu\\
	&\quad=\int_M \big(\big|\N^2\phi\big|^2+\Rc^f(\N\phi,\N\phi)-2\big\langle\Rc^{H,f},\N^2\phi\big\rangle\big) {\rm d}\mu +\frac12\int_M \big\langle\N\rho,\N R^{H,f}\big\rangle{\rm e}^{-f} {\rm d}V\\
 &\phantom{\quad={}}{}+\int_M \IP{\frac12\divg H^2-\frac1{4}\N\wnorm{H},\N \phi} {\rm d}\mu
	-\frac12\int_M H^2(\N f,\N\phi) {\rm d}\mu,
	\end{align*}
	which is the claim.
	
	For the last claim, we simply compute
	\begin{align*}
	\frac{{\rm d}}{{\rm d}t}\int_M R^{H,f} {\rm d}\mu=&\int_M \del_t R^{H,f}+\big\langle\N R^{H,f},\N\phi\big\rangle+\big(R^{H,f}\big)^2-\big(R^{H,f}\big)^2 {\rm d}\mu\\
	=&\int_M \del_t R^{H,f}+\big\langle\N R^{H,f},\N\phi\big\rangle {\rm d}\mu.\tag*{\qed}
	\end{align*} \renewcommand{\qed}{}
\end{proof}

\begin{prop} \label{prop:entropyconvexity} Fix $(g_t, H_t, f_t)$ a solution to generalized Ricci flow, and suppose $(\rho_t,\phi_t)$ solves the geodesic equations \eqref{eq: geodesic}. Then
 \begin{align*}
	\frac{{\rm d}^2}{{\rm d}t^2}\int_M \log\rho {\rm d}\mu
	=&\int_M \left(\big|\Rc^{H,f-\phi}\big|^2 + \Rc^{H,f}(\N\phi,\N\phi)+\frac12\del_t R^{H,f}+\big\langle\N R^{H,f},\N\phi\big\rangle \right) {\rm d}\mu.
	\end{align*}
 Also
	\begin{align*}
	\frac{{\rm d}^2}{{\rm d}t^2}\int_M (\log\rho-\phi) {\rm d}\mu
	=&\int_M \big|\Rc^{H,f-\phi}\big|^2 {\rm d}\mu.
	\end{align*}
\end{prop}
\begin{proof}
We obtain from Lemma \ref{lem: evolution entropy 2}
\begin{align*}
 \frac{{\rm d}^2}{{\rm d}t^2}\int_M \log\rho {\rm d}\mu={}&\int_M \big(|\N^2\phi|^2+\Rc^f(\N f,\N f)-2\big\langle\Rc^{H,f},\N^2\phi\big\rangle \big) {\rm d}\mu\\
	&+\int_M \IP{\frac12\divg H^2-\frac1{4}\N\wnorm{H},\N\phi} {\rm d}\mu\\
	&-\frac12\int_M H^2(\N f,\N\phi) {\rm d}\mu+\frac12\int_M \big\langle\N\rho,\N R^{H,f}\big\rangle{\rm e}^{-f} {\rm d}V\\
	&+\int_M \big(\del_t R^{H,f}+\big\langle\N R^{H,f},\N \phi\big\rangle \big) {\rm d}\mu.
\end{align*}
Using Proposition \ref{p:scalarmonotonicity}, Lemma \ref{l:HBianchi} and noting
\begin{align*}
	\big|\Rc^{H,f-\phi}\big|^2={}&\big|\Rc^{H,f}\big|^2+ \big|\N^2\phi\big|^2-2\big\langle\Rc^{H,f},\N^2\phi\big\rangle
	\\&-\frac12\IP{d^*H+i_{\N f}H,i_{\N\phi}H}+\frac14 \brs{i_{\N\phi}H}^2,
	\end{align*}
 we find
 \begin{align*}
 \frac{{\rm d}^2}{{\rm d}t^2}\int_M \log\rho {\rm d}\mu
	=&\int_M \big(\big|\Rc^{H,f-\phi}\big|^2 + \Rc^{H,f}(\N\phi,\N\phi)+\frac12\del_t R^{H,f}+\big\langle\N R^{H,f},\N\phi\big\rangle \big) {\rm d}\mu,
\end{align*}
as claimed. The second claim of the proposition then follows easily from Lemma \ref{lem: evolution entropy 1}.
\end{proof}

\subsection{Cost monotonicity}

Given the setup as above, for $\mu',\mu''\in P^\infty(M)$ define the distance
\begin{align*}
 C_0^{t',t''}(\mu',\mu''):=\inf_\mu E_0^{t',t''}(\mu),
\end{align*}
where the infimum is taken among all paths of smooth measures $\mu:=\rho {\rm e}^{-f} {\rm d}V\colon [t',t'']\to P^\infty(M)$ with $\mu(t')=\mu'$ and $\mu(t'')=\mu''$, and such that (\ref{eq: conteq}) is satisfied.

\begin{prop} \label{prop: cost}
	Let $\mu\colon[t',t'']\times (-\eps,\eps)\to P^\infty(M)$ be a smooth map, where $\mu=\mu(t,u)$. Define~${\mu_u\colon[t'+u,t''+u]\to P(M)}$ by $\mu_u(t):=\mu(t-u,u)$. Suppose that $\mu_0=\mu(\cdot,0)$ is a~minimizer for $E_0^{t',t''}$, i.e., there exists $\phi_0=\phi(\cdot,0)$ such that \eqref{eq: geodesic} holds. Then
	\begin{align*}
	\left. \frac{{\rm d}}{{\rm d}u} \right|_{u=0}E_0^{t'+u,t''+u}(\mu_u)={}
	&\int_{t'}^{t''}\int_M \big|\Rc^{H,f-\phi_0}\big|^2 {\rm d}\mu_0{\rm d}t \\
 &+\left. \int_M \phi\big(\left.\del_u\right|_{u=0}\rho(\cdot,u)-R^{H,f}\rho_0+\Delta_f \rho_0\big){\rm e}^{-f} {\rm d}V\right|_{t=t'}^{t''}.
	\end{align*}
\end{prop}

\begin{proof}
Note that we express
\begin{align*}
E_0^{t'+u,t''+u}(\mu_u)=\frac12\int_{t'}^{t''}\int_M \big(|\N\phi(t,u)|^2+R^{H,f}\big)\rho(t,u){\rm e}^{-f} {\rm d}V{\rm d}t,
\end{align*}
where the metric, volume and $f$ are evaluated at time $t+u$. Then we compute
\begin{align}
\left. \frac{{\rm d}}{{\rm d}u} \right|_{u=0} E_0^{t'+u,t''+u}(\mu_u)={}&\int_{t'}^{t''}\int_M \left[ \Rc^{H,f}(\N\phi,\N\phi) + \IP{\N\phi,\N\del_u\phi}+\frac12\del_t R^{H,f} \right] {\rm d}\mu {\rm d}t\nonumber\\
& +\frac12\int_{t'}^{t''}\int_M \big(|\N\phi|^2+R^{H,f}\big)\big(\del_u\rho-R^{H,f}\rho\big) {\rm e}^{-f} {\rm d}V{\rm d}t,\label{eq: e}
\end{align}
where $\mu(t,u)$, $\rho(t,u)$ and $\phi(t,u)$ are evaluated at $u=0$.
For each $\psi\in C^{\infty}(M)$, by equation~\eqref{eq: conteq} we have
\begin{align*}
\int_M \psi\del_t\rho {\rm e}^{-f} {\rm d}V=\int_M \big(\IP{\N\psi,\N\phi}+\psi R^{H,f}\big) {\rm d}\mu.
\end{align*}
Hence
\begin{align}
\int_M \psi\big(\del_u\del_t\rho-R^{H,f}\del_t\rho\big){\rm e}^{-f} {\rm d}V
={}&\int_M \big(2\Rc^{H,f}(\N\psi,\N\phi) + \IP{\N\psi,\N\del_u\phi}+\del_tR^{H,f}\psi\big) {\rm d}\mu\nonumber\\
& + \int_M \big(\IP{\N\psi,\N\phi}+R^{H,f}\psi\big)\big(\del_u\rho-R^{H,f}\rho\big){\rm e}^{-f} {\rm d}V.\label{eq: phi}
\end{align}
Combining
\eqref{eq: e} with \eqref{eq: phi}, we obtain after choosing $\psi=\phi$
\begin{align*}
\left. \frac{{\rm d}}{{\rm d}u} \right|_{u=0} E_0^{t'+u,t''+u}(\mu_u)={}&\int_{t'}^{t''}\int_M \phi\big(\del_u\del_t\rho-R^{H,f}\del_t\rho\big){\rm e}^{-f} {\rm d}V{\rm d}t\\
&+\int_{t'}^{t''}\int_M \left[ - \Rc^{H,f}(\N\phi,\N\phi) + \frac{1}{2} \del_t R^{H,f} - \del_t R^{H,f} \phi \right] {\rm d}\mu {\rm d}t\\
&-\int_{t'}^{t''}\int_M \left( \frac12|\N\phi|^2-\frac12 R^{H,f} \right)\big(\del_u\rho-R^{H,f}\rho\big) {\rm e}^{-f} {\rm d}V{\rm d}t\\
&-\int_{t'}^{t''}\int_M R^{H,f}\phi\big(\del_u\rho-R^{H,f}\rho\big) {\rm e}^{-f} {\rm d}V{\rm d}t.
\end{align*}
Integrating by parts in $t$, we have
\begin{align*}
&\int_{t'}^{t''}\int_M \phi\big(\del_u\del_t\rho-R^{H,f}\del_t\rho\big){\rm e}^{-f} {\rm d}V{\rm d}t\\
&\quad=\left.\int_M \phi\big(\del_u\rho-R^{H,f}\rho\big){\rm e}^{-f} {\rm d}V\right|_{t=t'}^{t''}+\int_{t'}^{t''}\int_M \big[-\del_t\phi\del_u\rho+\phi\del_u\rho R^{H,f} \big]{\rm e}^{-f} {\rm d}V{\rm d}t\\
&\phantom{\quad={}}{}+\int_{t'}^{t''}\int_M \big[\del_t R^{H,f}\phi+\del_t\phi R^{H,f}-\big(R^{H,f}\big)^2\phi \big] {\rm d}\mu {\rm d}t,
\end{align*}
thus yielding
\begin{align*}
\left. \frac{{\rm d}}{{\rm d}u} \right|_{u=0} E_0^{t'+u,t''+u}(\mu_u)=&\left.\int_M \phi\big(\del_u\rho-R^{H,f}\rho\big){\rm e}^{-f} {\rm d}V\right|_{t=t'}^{t''} \\
&+ \int_{t'}^{t''}\int_M \left[ - \Rc^{H,f}(\N\phi,\N\phi) + \frac{1}{2} \del_t R^{H,f} \right] {\rm d}\mu {\rm d}t.
\end{align*}
We know from Lemma \ref{lem: evolution entropy 2} that
\begin{align*}
	\frac{{\rm d}}{{\rm d}t} \int_M \IP{\N\rho,\N\phi}{\rm e}^{-f} {\rm d}V={}&\int_M \big(\big|\N^2\phi\big|^2+\Rc^f(\N\phi,\N\phi)-2\big\langle\Rc^{H,f},\N^2\phi\big\rangle\big) {\rm d}\mu\\
	&+\int_M \IP{\frac12\divg H^2- \frac1{4}\N\wnorm{H},\N \phi} {\rm d}\mu \\
	&-\frac12\int_M H^2(\N f,\N\phi) {\rm d}\mu+\frac12\int_M \big\langle\N\rho,\N R^{H,f}\big\rangle{\rm e}^{-f} {\rm d}V.
	\end{align*}
	Inserting this and using the result of Proposition \ref{p:scalarmonotonicity} and Lemma \ref{l:HBianchi} gives
	\begin{align*}
\left. \frac{{\rm d}}{{\rm d}u} \right|_{u=0} E_0^{t'+u,t''+u}(\mu_u)={}&\left.\int_M \phi\big(\del_u\rho+\Delta \rho-\IP{\N\rho,\N f}-R^{H,f}\rho\big){\rm e}^{-f} {\rm d}V\right|_{t=t'}^{t''}\\
& +\int_{t'}^{t''}\int_M \big|\Rc^{H,f}-\N^2\phi\big|^2 {\rm d}\mu {\rm d}t \\
& +\int_{t'}^{t''}\int_M \IP{\frac12\divg H^2-\frac1{4}\N\wnorm{H},\N\phi} {\rm d}\mu {\rm d}t\\
& +\int_{t'}^{t''}\int_M \frac14 H^2(\N\phi,\N\phi)-\frac12 H^2(\N f,\N\phi) {\rm d}\mu {\rm d}t\\
={}&\left.\int_M \phi\big(\del_u\rho+\Delta \rho-\IP{\N\rho,\N f}-R^{H,f}\rho\big){\rm e}^{-f} {\rm d}V\right|_{t=t'}^{t''} \\
& +\int_{t'}^{t''}\int_M \big|\Rc^{H,f-\phi}\big|^2 {\rm d}\mu {\rm d}t,
\end{align*}
as claimed.
\end{proof}

Using this, we establish monotonicity of the cost along the backwards heat flow, and use it to obtain the monotonicity of the energy functional along generalized Ricci flow.

\begin{cor}\label{cor: cost} Under the hypothesis of Proposition {\rm \ref{prop: cost}}, suppose that each $\mu_u$ is a minimizer for $E_0^{t'+u,t''+u}$. Suppose that the endpoint measures $\mu_u(t'+u)=\mu(t',u)$ and $\mu_u(t''+u)=\mu(t'',u)$ satisfy the backward heat equation \eqref{eq: heatmeasure} in $u$.
Then
\begin{align*}
 u\mapsto C_0^{t'+u,t''+u}(\mu_u(t'+u),\mu_u(t''+u))
\end{align*}
is nondecreasing.
\end{cor}

\begin{cor} \label{c:Fmonotone} Suppose that $\mu\subset P^\infty(M)$ is a smooth solution of the backward heat equation~\eqref{eq: heatmeasure}. Then
\begin{align*}
 \mathcal F=\int_M\big[\brs{\N\log\rho}^2+R^{H,f}\big] {\rm d}\mu
\end{align*}
is nondecreasing in $t$.
\end{cor}

\begin{proof}
 Fix a time $t'$. Using ellipticity of the linearized geodesic equation and an inverse function theorem argument, for $t''$ all sufficiently close to $t'$ and $u > 0$ sufficiently small, the minimizing geodesic connecting $\mu(t' + u)$ and $\mu(t'' + u)$ is smooth. By Lemma \ref{l:rhoBHE} and Corollary \ref{cor: cost}, we have
 \begin{align*}
 \frac{C_0^{t'+u,t''+u}(\mu(t'+u),\mu''(t''+u))}{t''-t'}\geq \frac{C_0^{t',t''}(\mu(t'),\mu(t''))}{t''-t'}.
 \end{align*}
 Letting $t''\to t'$
 \begin{align*}
 \left. \frac12\int_M\big[\brs{\N\phi}^2+R^{H,f}\big] {\rm d}\mu\right|_{t'+u}\geq
 \left. \frac12\int_M\big[\brs{\N\phi}^2+R^{H,f}\big] {\rm d}\mu\right|_{t'}.
 \end{align*}
 As $\rho$ solves \eqref{eq: heatdensity} it follows that $\N \phi= \N \log \rho$, giving the claim.
\end{proof}

\subsection*{Acknowledgements}

We thank Micah Warren for helpful comments. The second named author was supported by a~Simons Fellowship and by the NSF via DMS-2203536.


\pdfbookmark[1]{References}{ref}
\LastPageEnding

\end{document}